\documentclass[11pt]{article}
\usepackage[final]{epsfig}
\usepackage[left=3cm,right=3cm, top=2.5cm,bottom=2.5cm,bindingoffset=0cm]{geometry}
\usepackage{graphics}
\usepackage{amsmath}
\usepackage{amsfonts}
\usepackage{latexsym}
\usepackage{amssymb, mathabx}
\usepackage{amsthm}
\usepackage{graphicx}
\usepackage{epstopdf}
\usepackage{multicol,multirow}
\usepackage{hyperref, enumitem}
\usepackage{mathtools}

 \usepackage{relsize, authblk}

\usepackage[nottoc]{tocbibind}

\usepackage{epigraph}
\mathtoolsset{showonlyrefs}

\usepackage{tikz}
\usetikzlibrary{positioning}
\usetikzlibrary{decorations.text}
\usetikzlibrary{decorations.markings}

\usetikzlibrary{arrows} 
\tikzset{
    >=stealth',
    pil/.style={
           ->,
           thick,
           shorten <=2pt,
           shorten >=2pt,}
}
\tikzset{->-/.style={decoration={
  markings,
  mark=at position .7 with {\arrow{>}}},postaction={decorate}}}
  \tikzset{a/.style={decoration={
  markings,
  mark=at position .52 with {\arrow{angle 90}}},postaction={decorate}}}
\tikzset{-<-/.style={decoration={
  markings,
  mark=at position .4 with {\arrow{<}}},postaction={decorate}}}

\makeatletter
\def\theoremhead@plain#1#2#3{%
  \theoremname{#1}\theoremnumber{\@ifnotempty{#1}{ }\@upn{#2}}%
  \theoremnote{ {\the\theorem@notefont#3}}}
\let\theoremhead\theoremhead@plain

\makeatother

\newcounter{AppCounter}

\def\restrict#1{\raise-.5ex\hbox{\ensuremath|}_{#1}}

\newtheorem{lemma}{Lemma}

\newtheorem{remark-definition}[lemma]{Remark-Definition}
\newtheorem{theorem}[lemma]{Theorem}

\newtheorem{proposition-conjecture}[lemma]{Proposition-conjecture}

\theoremstyle{definition}

\newtheorem{remark}[lemma]{Remark}

\newcommand{\R}{\mathbb R}
\newcommand{\C}{\mathbb C}

\begin{document}

\title{The Golden Ratio and Hydrodynamics}

\date{~}



\author[*]{Boris Khesin$^*$ and Hanchun Wang}
\affil[*]{Dept of Math, Univ of Toronto, ON M5S 2E4, Canada}

\maketitle

\vspace{-0.7in}
\epigraph{He [Radishchev] wanted simultaneously to write a subtle, graceful, and witty
prose, but also to serve his fatherland... For mixing the genres Radishchev
got jailed for ten years.}{Pyotr Vail and Alexander Genis, \textit{Native Speech}}



\begin{abstract}
There are useful and useless golden ratios. The useful one helps in traffic. The useless and rather mysterious one arises in hydrodynamics of point vortices, which we discuss in detail.
\end{abstract}



\bigskip



\subsection*{Mysterious Golden Ratios}

There are useful and useless golden ratios. We start with the former. 

The best way to convert miles to kilometers  is to take the next Fibonacci number. Indeed, 5 mi is 8 km, or 55 mi/h (the most important speed limit in the USA until 1995) is 89 km/h. Conversely, take the previous Fibonacci number: 34 km is 21 mi, or 
130 km/h (the speed limit in France) corresponds to 80 mi/h (one can always squeeze in extra zeros to the Fibonacci pair 13 and 8); see Figure \ref{speedometer}.

A justification for this simple rule is that in the Fibonacci sequence 
$$
1,1,2,3,5,8,13,21,34,55,89,144,\dots
$$
the ratio of two consecutive terms tends to the golden ratio (or golden section) 
$$
\phi=\frac{\sqrt 5 +1}{2}= 1.618034...
$$ as we go along the sequence.  The latter differs from the mile/kilometer factor $mi/km=1.6093$ by about $0.5\%$ (which is better than the precision of the police radar detector), hence the above mnemonic rule.

\medskip

There are many equivalent definitions of the golden ratio. The basic one is that this is the ratio of length to width for a rectangle, which preserves this ratio after cutting out the square; see Figure  \ref{rectangle}. Symbolically, $\phi:=a/b=
b/(a-b)$, which implies that $\phi$ is the positive root of the quadratic equation $\phi^2-\phi-1=0$.


\begin{figure}[!htb]
   \begin{minipage}{0.5\textwidth}
     \centering
\includegraphics[width=2.8in]{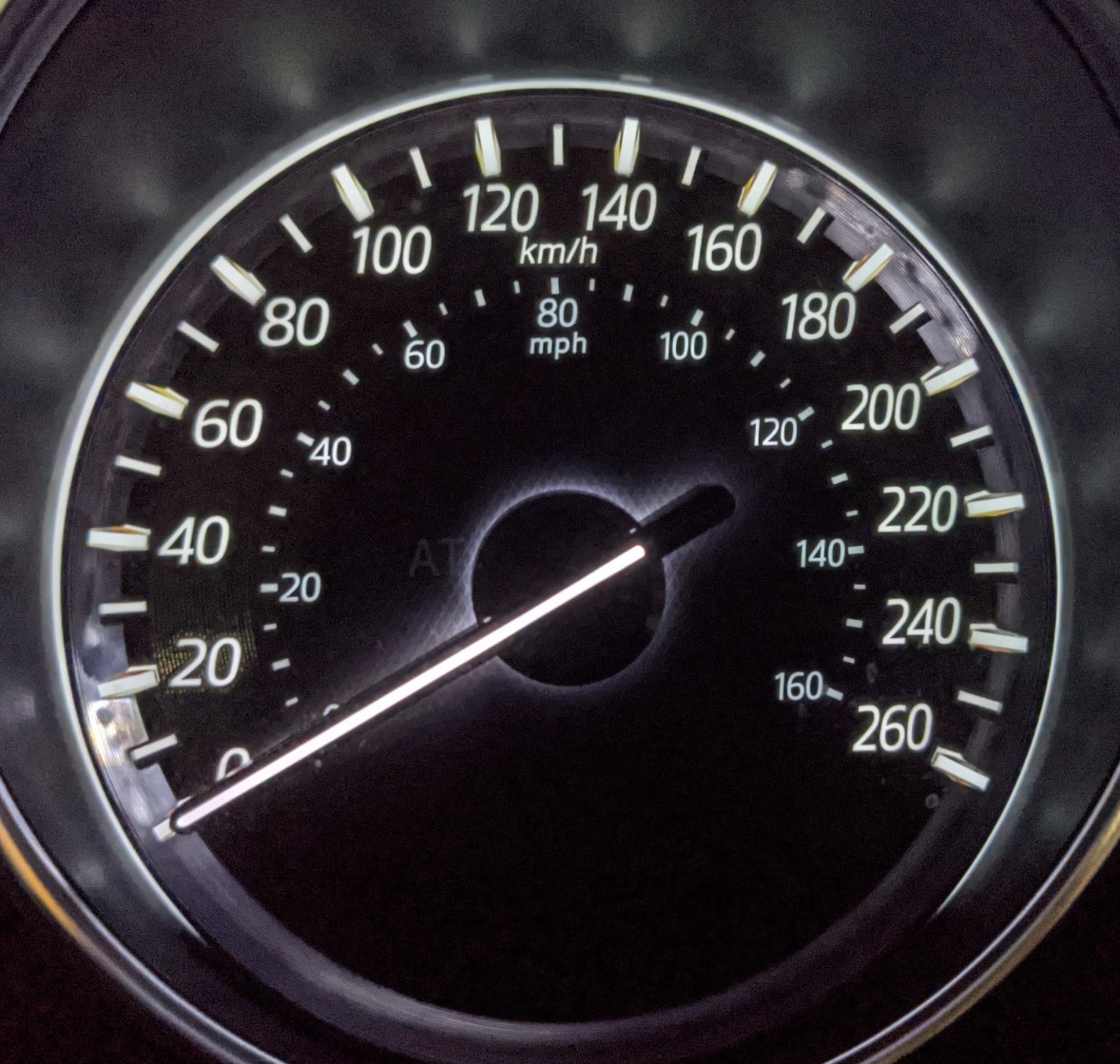}
\caption{\small The Mazda CX-5 speedometer with both $km/h$ and $mph$ scales.}
\label{speedometer}
   \end{minipage}\hfill
   \begin{minipage}{0.38\textwidth}
     \centering
\includegraphics[width=2.2in]{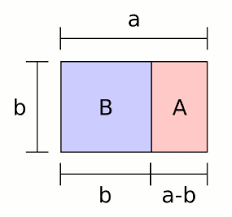}
\caption{\small The golden ratio $\phi=a/b$ is the ratio of length to width for such a special rectangle.}
\label{rectangle}
   \end{minipage}
\end{figure}

There are plenty of appearances and applications of the golden ratio in the medieval architecture, phyllotaxis, mollusk 
shells, one-dimensional dynamics, etc. Here we describe the recently discovered, rather mysterious, and, likely, most useless appearance of the golden ratio in 2D hydrodynamics.

\bigskip


\subsection*{Dynamics of point vortices}

The evolution of the earth atmosphere or oceans is often thought of as a motion of an inviscid incompressible
two-dimensional  fluid.  The classical Euler equation for such a motion can be written 
in the vorticity form $\partial_t \omega +L_v\omega=0$, describing that the fluid's vorticity $\omega:={\rm curl}\, v$ is transported by the fluid flow with velocity $v$. In two dimensions the vorticity $\omega$ is a function, and the Euler equation
assumes the form 
$$
\partial_t \omega =\frac{\partial \psi}{\partial x}\frac{\partial \omega}{\partial y}
-\frac{\partial \psi}{\partial y}\frac{\partial \omega}{\partial x}\,,
$$
where  the stream function $\psi$ of the flow satisfies 
$\Delta \psi=\omega$. 

This infinite-dimensional system in turn can be viewed as a limit of 
 dynamics of so-called point (or singular) vortices. A description of the 
 corresponding finite-dimensional dynamical system of vortices
in the plane goes back to Helmholtz and Kirchhoff; see e.g. \cite{Helmholtz, Kirchhoff, AK}. 
Namely, let the vorticity $\omega$ be  supported on $N$ point vortices $\omega=\sum^N_{j=1} \Gamma_j\,\delta( z- z_j)$, 
where $ z_j=( x_j,  y_j)$ are coordinates and $\Gamma_j$ is the strength of the $j$th point vortex in the plane $\R^{2}=\C$. Then 
the evolution of $N$  vortices according to the Euler equation is described by the  system
$$
\Gamma_j \dot x_{j}
=\frac{\partial \mathcal H}{\partial y_{j}},\qquad
\Gamma_j \dot y_{j}
=-\frac{\partial \mathcal H}{\partial x_{j}},\qquad 
1\le j\le N\,,
$$
for the function 
$$
\mathcal H(z_1, ..., z_N)=-\frac{1}{4\pi}\sum^N_{j<k}\Gamma_j\Gamma_k \, 
\ln | z_j- z_k |
$$
on $(\R^{2})^N$.
This system first appeared in this modern form in the 1876 Berlin lectures \cite{Kirchhoff} by Gustav Kirchhoff (1824-1887), 
who also found the system's three first integrals, related to its invariance with respect to the three-dimensional group $E(2)$ of motions of the plane. 

\begin{remark} 
The function $\mathcal H$ is the system's Hamiltonian
for the standard Poisson  bracket  on $\R^{2N}$ weighted by the strengths $\Gamma_j$.  
In addition to the Hamiltonian function, there are two first integrals in involution that are coming from Noether's symmetries.
This implies  the integrability for small number of vortices: it turns out that the systems of $N=1, 2$ and 3 point vortices in the plane $\R^2$  are 
completely integrable, while the systems of $N \ge 4$ point vortices for generic strengths are not; see e.g. \cite{AK}.

The cases of one and two point vortices were already studied in detail by Hermann von Helmholtz (1821-1894) some twenty years earlier,  in his 1858 paper \cite{Helmholtz}. He discovered that a single point vortex ($N=1$)  in the plane stays at rest, while a pair of point vortices ($N=2$)  rotates about their common center of vorticity  $z_c:=(\Gamma_1 z_1+ \Gamma_2 z_2)/(\Gamma_1+ \Gamma_2)$. For instance, the vortex pair with $\Gamma_1= \Gamma_2$
rotates about its midpoint, Figure \ref{dolzh_fig}a. For generic $\Gamma_1, \Gamma_2$ the rotation center is located on the line joining the vortices, and it is between the vortices provided that their strengths are of the same sign, and outside the segment joining them  if they are of different signs. In the case where the vortices form a dipole, i.e., 
a pair with $\Gamma_1=- \Gamma_2$,  they move uniformly along its perpendicular bisector, Figure \ref{dolzh_fig}b, 
and this can be regarded as a rotation about a center at infinity.
(The case of two point vortices is of particular interest in meteorology; see \cite{Dolzh}. For instance, the dipole setting corresponds to the cyclone--anticyclone pair moving in the zonal direction; see Figure \ref{cyclone}.)

\begin{figure}[hbtp]
\centering
\includegraphics[width=4.4in]{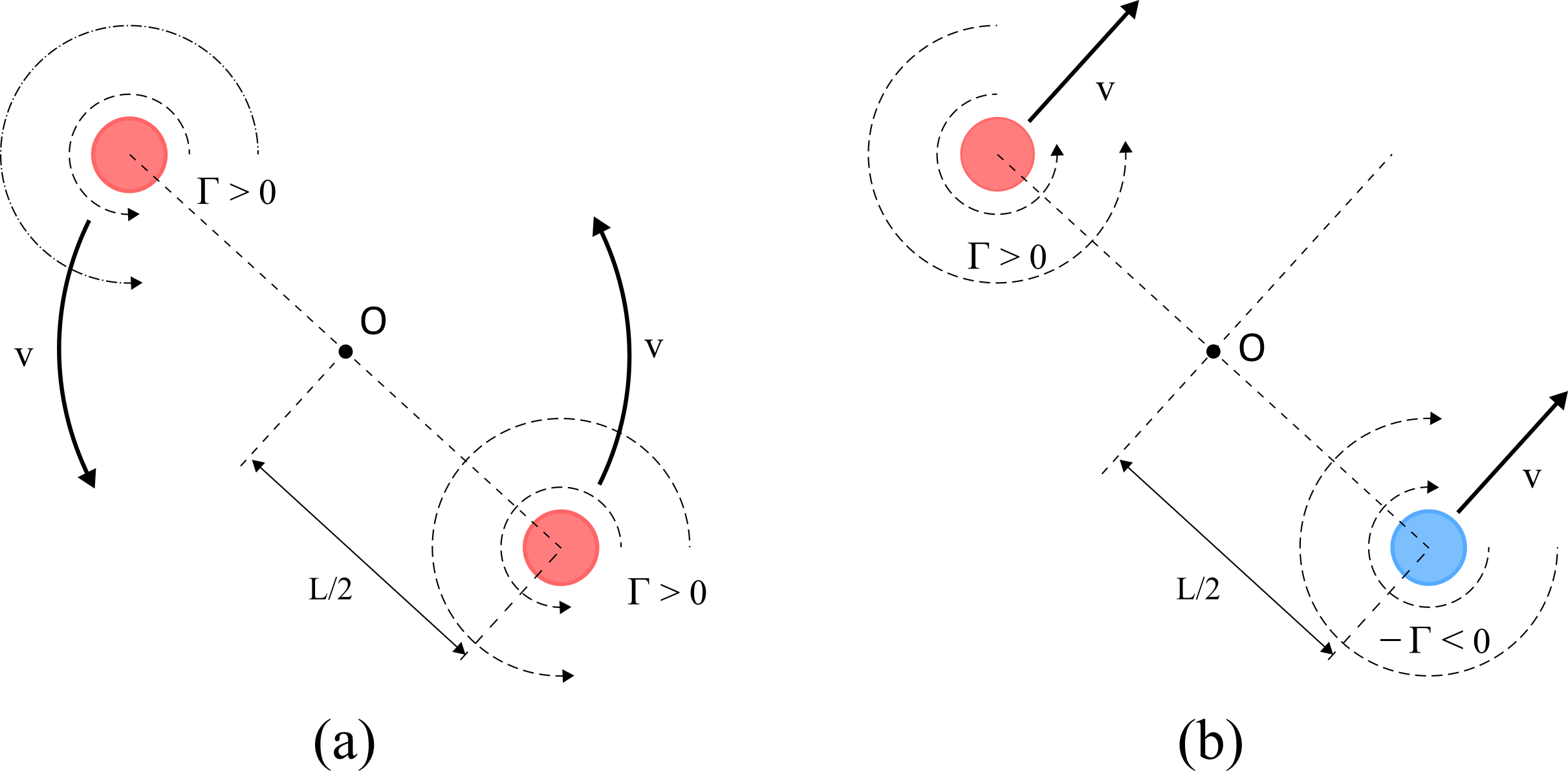}
\caption{\small Systems of two vortices: a) a pair of point vortices of equal strengths
rotates about their center, b) a dipole moves with constant velocity along its perpendicular bisector.}
\label{dolzh_fig}
\end{figure}

\begin{figure}[hbtp]
\centering
\includegraphics[width=4in]{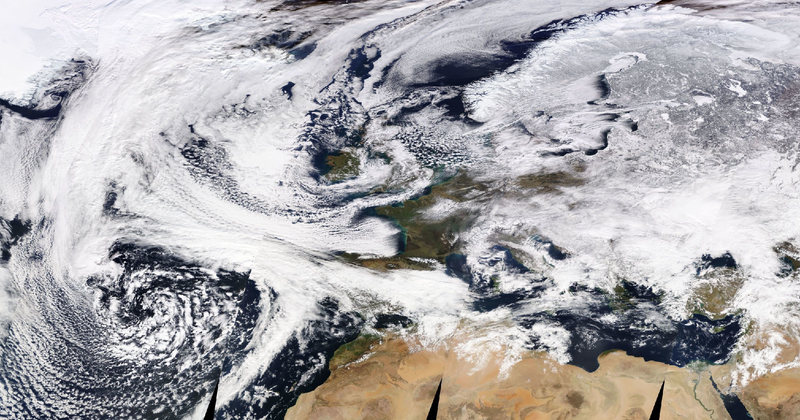}
\caption{\small The interaction of Cyclone Emma (nearing from the SW) and Anticyclone Hartmut (covering Europe from the NE) on February 27, 2018 (NASA, Wiki-Commons).}
\label{cyclone}
\end{figure}

The motion of three point vortices ($N=3$) is still integrable but can be much more elaborate, and
it was  described in detail by Walter Gr\"{o}bli in his dissertation in 1877 \cite{Grobli}. Gr\"obli was a student of Heinrich Weber and took courses of Kirchhoff, Helmholtz, Kummer, and Weierstrass in Berlin; see more details to this story in \cite {Aref}. 
In particular, he discovered that unlike the two-vortex case, three vortices allow self-similar collapsing solutions; see Figure~\ref{grobli_fig}.
(Note that  point vortices can never have pairwise collisions, even in many-vortex motions, since when only two point vortices approach each other, their interaction with other vortices is weakening, and they start behaving like a two-vortex system, therefore approximately rotating around their common vorticity center.) 

It is worth mentioning that unlike the famous three-body problem, which requires specifying both positions and velocities as its initial conditions, the three-vortex problem requires only vortex positions to determine the resulting trajectory.
\end{remark}

\begin{figure}[hbtp]
\centering
\includegraphics[width=3.2in]{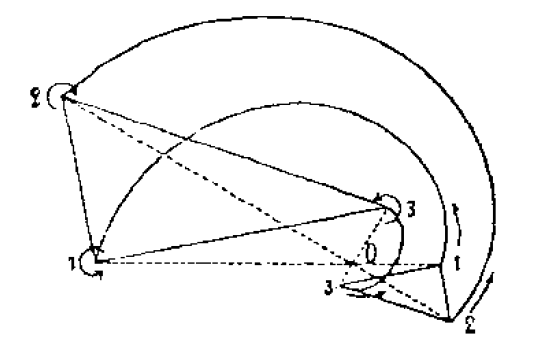}
\caption{\small Self-similar motion in which the vortex triangle changes its size but not its shape (a drawing from W.~Gr\"obli's 1877 dissertation). This self-similar expansion corresponds to vortices of strengths $\Gamma_1=3, \Gamma_2=-2$ and $\Gamma_3=6$; see \cite{Aref}.}
\label{grobli_fig}
\end{figure}

\begin{remark}
It is interesting to trace how the motion of vortices changes on domains with boundary.  
Already the half-plane brings a wider variety to this classical problem. Indeed, 
the motion of point vortices in a half-plane can be obtained by considering the auxiliary  reflected point vortices, so that 
the Green function for the Laplacian $\Delta \psi = \omega$ on the half-plane satisfies the zero boundary condition. An interesting feature of this motion is that, while a single vortex in the whole plane does not move, a single vortex is the half-plane moves parallel to the boundary with constant speed, in a sense ``interacting with the boundary". 
Indeed, its motion is equivalent to the motion of a dipole in $\R^2$ with auxiliary point vortex of opposite strength
reflected in the boundary, while the boundary becomes the perpendicular bisector for such a pair.

For two point vortices in the half-plane one observes many types of motions, some of which are of particular interest. For instance, it allows a leapfroging motion along the border, which is also related to the famous leapfrogging motion of axisymmetric vortex rings in 3D. We detail the motion of two point vortices in a half-plane  in the next section. 

It is worth pointing out that the system in the half-plane 
is invariant only under translations parallel to the boundary, and hence has only one 
Noether invariant. Thus this Hamiltonian system of point vortices in the half-plane is integrable  for $N=1$ and 2, while for $N=3$ it was recently proven to be non-integrable \cite{Cheng}.
\end{remark}


\bigskip


\subsection*{A pair of point vortices in the half-plane}

Here we will study the motion of two point vortices in the half-plane for two special but representative cases: a vortex pair or a dipole, which are two point vortices of the same absolute strengths and having, respectively, either the same or the opposite signs.

For two point vortices of any strengths $\Gamma_1$ and $\Gamma_2$ located 
at $z_1 = (x_1, y_1)$ and $z_2 = (x_2, y_2)$ in the half-plane the corresponding Hamiltonian
takes into account their interaction between themselves and with their mirror images $\bar z_1 = (x_1, -y_1)$ and $\bar z_2 = (x_2, -y_2)$ of strengths $-\Gamma_1$ and $-\Gamma_2$:
$$
\mathcal H(z_1, z_2)=\frac{1}{4\pi}\left( \Gamma_1^2\log  | z_1- \bar z_1|+\Gamma_2^2\log | z_2- \bar z_2|-2\Gamma_1\Gamma_2\log | z_1- z_2| +2 \Gamma_1\Gamma_2\log | z_1- \bar z_2| \right)\,.
$$
The Noether first integral is $P:=\Gamma_1 y_1+ \Gamma_2 y_2$, which corresponds to the system symmetry with respect to the $x$-translations. 

To study vortex bifurcations we normalize their strengths by setting $\Gamma_1=\Gamma_2=1$ for a vortex pair and $\Gamma_1=-\Gamma_2=1$
for a dipole.
Furthermore, introduce the following dimensionless parameter $W:=P^2\exp(-4\pi\mathcal H)$ 
measuring the vortex interaction. As we will see below, the increase of $W$ corresponds to the 
weakening of the interaction between the point vortices.

\begin{remark}
To see that this is indeed a dimensionless parameter, one observes that whatever units of length are used 
for $x$ and $y$, the value of $P^2$ is measured in $units^2$, while $\exp(-\mathcal H)$ is measured 
in $units^{-2}$, so that $W$ is dimensionless. More generally, before normalizations,
one can define dimensionless $W:=(P/\Gamma)^{2}\exp(-4\pi\mathcal H/\Gamma^2)$, where $\Gamma:=\Gamma_1=\pm\Gamma_2$.
\end{remark}

By changing this parameter one observes different types of behaviour for the corresponding vortex systems. As $W$ increases, the  interaction of vortices weakens: 
in a dipole, upon approaching one another 
the vortices either go to infinity at a slanted asymptote, or one of them makes a kink with the other before parting, or they pass each other at a social distance; see Figure \ref{gold_dipole_fig}.

\begin{figure}[hbtp]
\hspace*{-2cm}     
\includegraphics[width=7.6in]{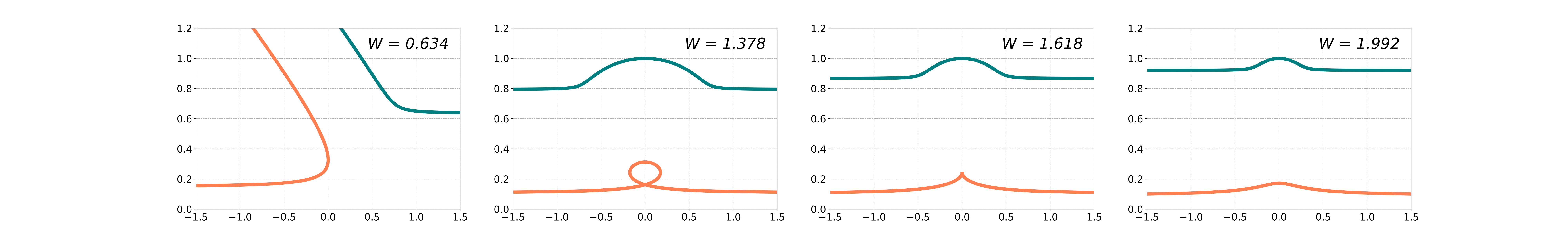}
\caption{\small As the interaction weakens, the vortices in the dipole either go to infinity, or one of them makes a kink, or they pass around each other. (The orange and aquamarine colors correspond to two  vortices 1 and 2).}
\label{gold_dipole_fig}
\end{figure}

For a vortex pair, a strong interaction is related to a leap-frogging motion of the vortices, while a weak one produces just two intertwining sinusoidal-like trajectories; see Figure \ref{gold_pair_fig}. In both cases the change from a kink/leap-frogging 
motion to a regular one happens through the cusp-type bifurcation for the vortex motion:

\begin{figure}[hbtp]
\hspace*{-2cm}     
\includegraphics[width=7.6in]{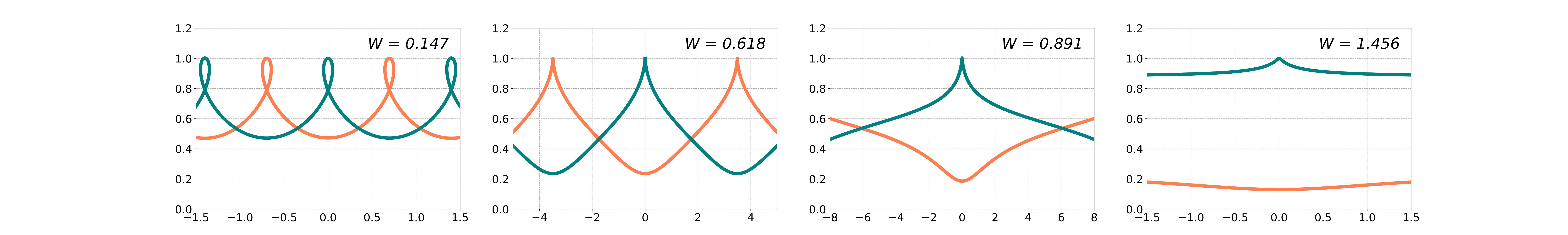}
\caption{\small For a vortex pair, as the interaction weakens, a leap-frogging motion of the vortices changes to intertwining sinusoidal-like trajectories via a cusp-type motion.}
\label{gold_pair_fig}
\end{figure}

\begin{theorem}\label{thm:3cases}
(i) For a vortex pair the leap-frogging vortex motion changes to the smooth intertwining one 
  through  the cusp bifurcation, which occurs at the reciprocal golden ratio value of its
 interaction parameter: $W=1/\phi$.
 
(ii) For a dipole the bifurcation from the kink motion to the laminar one
occurs through the cusp bifurcation, corresponding to the golden ratio value of its
 interaction parameter: $W=\phi$.

(iii) At the moment of cusp bifurcation the two vortices lie
on the same vertical. The cross-ratio of the four points $(z_1, z_2, \bar z_2, \bar z_1)$, that are these two vortices 
together with their mirror images, is equal to the golden ratio $\phi$ (in both cases, for a vortex pair and a dipole).
\end{theorem}

Before we prove this theorem we make several  observations for general strengths $\Gamma_1$ and $\Gamma_2$
of two vortices in the half-plane. First of all, note that the cusp in a generic parametrized curve $t\mapsto (x(t), y(t))$
in the plane  corresponds to the vanishing derivative, $(\dot  x(t), \dot y(t))=(0,0)$, 
an instantaneous  stop. (Indeed, a generic curve 
$(x(t), y(t))=(a_0 +a_2 t^2 + a_3 t^3 +..., b_0 +b_2 t^2 + b_3 t^3 +...)$ with vanishing derivative at $t=0$ can be transformed by local coordinate changes to 
$(x(t), y(t))=(t^3, t^2)$, a normal form of a cusp.)

Let us rewrite the general two-vortex Hamiltonian in a more explicit coordinate form:
$$
\mathcal H(z_1, z_2)=\frac{1}{4\pi}\log \left( (2y_1)^{\Gamma_1^2} (2y_2)^{\Gamma_2^2} 
\left(\frac{(x_1-x_2)^2 +(y_1+y_2)^2}{(x_1-x_2)^2 +(y_1-y_2)^2}    \right)^{\Gamma_1\Gamma_2} \right)\,.
$$
Recall one useful property of such a
system: it does not allow collisions of vortices between themselves and with the boundary \cite{Cheng}. 
We start with the following auxiliary statement.

\begin{lemma}
For a pair of vortices in the half plane, the velocity of either of the vortices is horizontal whenever the two vortices turn out to be on the same vertical, $x_1=x_2$.
\end{lemma}

\begin{proof}
For instance  for vortex $z_1$ we obtain:
$$
\dot y_1=-\frac{1}{\Gamma_1} \frac{\partial \mathcal H}{\partial x_1}
=-\frac{\Gamma_2}{4\pi}\left( \frac{2(x_1-x_2)}{(x_1-x_2)^2+(y_1+y_2)^2}-\frac{2(x_1-x_2)}{(x_1-x_2)^2+(y_1-y_2)^2 } \right)
$$
$$
=\frac{2\Gamma_2}{\pi} \frac{(x_1-x_2)y_1y_2}{\left((x_1-x_2)^2+(y_1+y_2)^2\right)\left((x_1-x_2)^2+(y_1-y_2)^2\right)}\,.
$$
Since the vortices  in the half-plane never hit the boundary,  one has $y_1 y_2\not=0$, which implies that the velocity 
of $z_1=(x_1, y_1)$ is horizontal, i.e., 
$\dot y_1=0$ iff $x_1-x_2=0$.
\end{proof}

Note at the moment when the line connecting the two vortices becomes vertical
it passes through their mirror images as well.

\begin{lemma}
At the moment of instantaneous stop, the heights of the vortices are related by a factor of $2\lambda+\sqrt{4\lambda^2+1}$, i.e., $y_1=(2\lambda+\sqrt{4\lambda^2+1})y_2$, where $\lambda:=\Gamma_2/\Gamma_1$.
\end{lemma}

\begin{proof}
Indeed, for the vortex $z_1$ we have:
$$
\dot x_1=\frac{1}{\Gamma_1} \frac{\partial \mathcal H}{\partial y_1}
=\frac{\Gamma_1}{4\pi y_1}+\frac{\Gamma_2}{4\pi}
\left( \frac{2(y_1+y_2)}{(x_1-x_2)^2+(y_1+y_2)^2}-\frac{2(y_1-y_2)}{(x_1-x_2)^2+(y_1-y_2)^2 } \right)\,.
$$
Upon plugging in the condition $x_1-x_2=0$, which ensures $\dot y_1=0$, we obtain
$$
\dot x_1=\frac{\Gamma_1}{4\pi y_1}-\frac{\Gamma_2 y_2}{\pi (y_1^2- y_2^2)}=\frac{1}{4\pi}\frac{\Gamma_1 y_1^2-\Gamma_1 y_2^2 - 4\Gamma_2 y_1 y_2}{y_1(y_1^2-y_2^2)}\,.
$$
The denominator can never be zero, due to the absence of collisions between vortices and with the boundary.
Equating the numerator to zero and solving the corresponding equation $\Gamma_1 y_1^2-\Gamma_1 y_2^2 - 4\Gamma_2 y_1 y_2=0$ for a positive root, since both $y_1$ and $y_2$ are positive, we come to $y_1/y_2=2\lambda+\sqrt{4\lambda^2+1}$ with $\lambda:=\Gamma_2/\Gamma_1$. 
\end{proof}

Now we will set $\Gamma_1=\pm \Gamma_2$ and prove the theorem for a dipole and a vortex pair.

\begin{proof}
For a dipole $\lambda=-1$, and hence $y_1=(-2+\sqrt{5})y_2$ and $y_1^2-y_2^2+4y_1y_2=0$. Then 
$$
W=P^2\exp(-4\pi\mathcal H)=\frac{(y_1-y_2)^2(y_1+y_2)^2}{4y_1y_2(y_1-y_2)^2}=\frac{y_1+y_2}{y_2-y_1}
=\frac{-1+\sqrt 5}{3-\sqrt 5}=\frac{1+\sqrt 5}{2}=\phi\,.
$$ 
The computation with $\lambda=1$ for a vortex pair  is similar.

The cross-ratio  of four points $(y_1, y_2, y_3, y_4)$ on a line 
is defined as  
$$
{\rm CR}(y_1, y_2, y_3, y_4)=\frac{(y_1-y_4)(y_2-y_3)}{(y_1-y_2)(y_3-y_4)}\,.
$$
One immediately obtains that, e.g., ${\rm CR}(2+\sqrt 5,\, 1, -1, -2-\sqrt 5)=(1+\sqrt 5)/2=\phi$.
\end{proof}

\begin{remark}
For a general strength ratio $\lambda:=\Gamma_2/\Gamma_1$ the dimensionless interaction
parameter is $W(\lambda):=|P/\Gamma|^{1+\lambda^2}\exp(-4\pi \mathcal H/\Gamma^2)$.
By plugging in it the relation of instantaneous stop from the two lemmas above, one can obtain 
the whole bifurcation diagram of this system. 
\end{remark}

Returning to the case of $\lambda=\pm1$ we recall
 that at the moment of instantaneous stop both point vortices, and hence their mirror images, lie on the same vertical. 
The cusp itself is related to the fact that at a specific  value of $W$, controlling the interconnection, a dominant interaction between the vortices gets overridden by the interaction with their mirrors. This balanced proportion of distances between the vortices on the vertical line leads to the golden ratio due to the following more general observation.

\medskip

Consider a point $E=1$ on the real axis and an action of three other points on it.
We will be looking for a point $A>1$ satisfying the following condition: its 
``action" on the point $E$ combined with the action of  its mirror point $-A$ on $E$ is balanced by the action of the mirror image $-E$ on $E$ itself.

\begin{lemma}
Assume that the action is inverse proportional to the distance between the points. 
Then the cross-ratio of points $(A, E, -E, -A)$ satisfying the balance of actions is equal to the golden ratio:
$$
{\rm CR}(A, E, -E, -A)=\phi\,.
$$
\end{lemma}
 
 \begin{proof}
 The balance of actions reduces to the equation
 $$
 \frac{1}{A-1} +\frac{1}{A+1}=\frac 12\,,
 $$
 where $1/2$ arises as the reciprocal of the distance between $E=1$ and $-E=-1$.
 Then $A$ satisfies the quadratic equation $A^2-4A-1=0$, and therefore $A=2\pm\sqrt 5$.
 By choosing the positive root we come to the same cross-ratio as 
 in Theorem \ref{thm:3cases}$(iii)$.
 \end{proof}
 
 For the system of point vortices  this ``action" is the speed that the vortices 
induce on one another, while their balance means the instantaneous stop. The induced speed is
indeed inverse proportional to the distance: 
if the vorticity is $\omega=\delta(z)$ on ${\mathbb C}={\mathbb R}^2$, 
then the stream function is $\psi = \Delta^{-1}\omega= C\ln |z|$, and hence the fluid velocity $v$ 
has the property $|v|=|{\rm sgrad}\,\psi| \sim 1/|z|$ at the distance $|z|$ from the origin.

While this explains the cross-ratio part of the  theorem, the appearance of the 
 golden ratio in the interaction parameter in this hydrodynamical context of Helmholtz singular vortices remains elusive.

\bigskip

 {\bf Acknowledgments.}
We are indebted to Klas Modin and Cheng Yang for helpful discussions.
This work was partially supported by an NSERC research grant.





\begin{thebibliography}{99}

\footnotesize

\bibitem{Aref}
H.~ Aref, N.~Rott, and H.~Thomann, {\it Gr\"obli's solution of the three-vortex problem.}
Annual Review of Fluid Mechanics, vol. 24 (1992), 1-21.

\bibitem{AK}
V. I. Arnold and B. A. Khesin, {\it Topological methods in hydrodynamics.} Springer, 1998, 374pp.

\bibitem{Dolzh}
F.V.~Dolzhansky, {\it  Fundamentals of geophysical hydrodynamics.} 
Encyclopaedia of Math. Sci.: Mathematical Physics, vol.103, Springer-Verlag, 2013, 272pp.
 
\bibitem{Grobli}
W.~Gr\"obli, {\it Specielle Probleme \"uber die
Bewegung geradliniger paralleler Wirbelf\"aden.} (1877),
Z\"urich: Z\"urcher und Furrer. 86 pp., Also published in Vierteljahrsschrift der
Naturforschenden Gesellschaft in Z\"urich 22: 37-81; 129-65.

\bibitem{Helmholtz}
H.~Helmholtz,  {\it On integrals of the
hydrodynamical equations which express
vortex-motion.} (1858), Transl. P.G.~Tait, 1867, in {\it Phil. Mag.} (4) 33: 485-512; 

\bibitem{Kirchhoff}
G.R.~Kirchhoff, {\it Vorlesungen \"uber Mathematische Physik, Mechanik.}  (1877),  Lecture
20. Leipzig: Teubner.


\bibitem{Cheng}
C.~Yang, {\it Vortex motion of the Euler and Lake equations}, preprint  arXiv:2009.12004, (2020), to appear on J. of Nonlinear Sci., 19pp.


\end{thebibliography}
\end{document}